\documentclass{SCIYA2017enOL}
\usepackage{amsfonts,amssymb}
\usepackage{amsmath}
\usepackage{graphicx}
\usepackage{graphics}
\usepackage{epstopdf}
\usepackage{epsfig}
\usepackage{color}
\usepackage{tikz}
\numberwithin{equation}{section}

\newcommand{\lam}{\lambda}

\newcommand{\e}{\epsilon}
\newcommand{\p}{\partial}

\online

\begin{document}
\ensubject{fdsfd}

\ArticleType{ARTICLES}
\SpecialTopic{Special Issue on Differential Geometry}
\SubTitle{In Memory of Professor Zhengguo Bai {\rm(1916}-{\rm2015)}}
\Year{2021}
\Month{February}%
\Vol{60}
\No{1}
\BeginPage{1} %
\DOI{10.1007/s11425-000-0000-0}
\ReceiveDate{October 14, 2020}
\AcceptDate{February 3, 2021}

\title[]{Nonexistence of NNSC-cobordism of Bartnik data}{Nonexistence of NNSC-cobordism of Bartnik data}

\author[1]{Leyang Bo}{windrunnerbly@pku.edu.cn}
\author[2]{Yuguang Shi}{ygshi@math.pku.edu.cn}

\AuthorMark{Leyang Bo,Yuguang Shi}

\AuthorCitation{Leyang Bo, Yuguang Shi}

\address[1]{Key Laboratory of Pure and Applied Mathematics, School of Mathematical Sciences, \\
Peking University, Beijing, 100871, P.\ R.\ China;}
\address[2]{Key Laboratory of Pure and Applied Mathematics, School of Mathematical Sciences,\\
 Peking University, Beijing, 100871, P.\ R.\ China;}

\abstract{
 In this paper, we consider the problem of nonnegative scalar curvature (NNSC) cobordism of Bartnik data $(\Sigma_1^{n-1}, \gamma_1, H_1)$ and $(\Sigma_2^{n-1}, \gamma_2, H_2)$. We prove that given two metrics $\gamma_1$ and $\gamma_2$ on $S^{n-1}$ ($3\le n\le 7$) with $H_1$ fixed, then $(S^{n-1}, \gamma_1, H_1)$ and $(S^{n-1}, \gamma_2, H_2)$ admit no NNSC cobordism  provided the prescribed mean curvature $H_2$ is large enough(Theorem \ref{highdimnoncob0}). Moreover, we show that for $n=3$, a much weaker condition that the total mean curvature $\int_{S^2}H_2d\mu_{\gamma_2}$ is large enough rules out NNSC cobordisms(Theorem \ref{2-d0}); if we require the Gaussian curvature of $\gamma_2$ to be positive, we get a criterion for non existence of trivial NNSC-cobordism by using Hawking mass and Brown-York mass(Theorem \ref{cobordism20}). For the general topology case, we prove that $(\Sigma_1^{n-1}, \gamma_1, 0)$ and $(\Sigma_2^{n-1}, \gamma_2, H_2)$ admit no NNSC cobordism provided the prescribed mean curvature $H_2$ is large enough(Theorem \ref{highdimnoncob10}).}



\maketitle

\section{Introduction}

Generalized Bartnik data $(\Sigma^{n-1}, \gamma, H)$ consists of an $(n-1)$-dimensional orientable  Riemannian manifold $(\Sigma^{n-1}, \gamma)$ and a smooth function $H$ defined on $\Sigma^{n-1}$ which  serves as the mean curvature of $\Sigma^{n-1}$. One natural question is to study nonnegative scalar curvature  (NNSC)-cobordism of Bartnik data $(\Sigma_i^{n-1}, \gamma_i, H_i)$, $i=1,2$.(cf \cite{HS}). Namely, {\it for  Bartnik data $(\Sigma^{n-1}_i, \gamma_i, H_i)$, $i=1,2$, we say $(\Sigma^{n-1}_1, \gamma_1, H_1)$ is NNSC-cobordant to $(\Sigma^{n-1}_2, \gamma_2, H_2)$ if there is an orientable $n$-dimensional manifold $(\Omega^n,g)$ with $\partial \Omega^n=\Sigma_1^{n-1} \cup \Sigma^{n-1}_2$,  and with the scalar curvature $R(g)\geq 0$, $\gamma_i=g|_{\Sigma_i}$, $i=1,2$, $H_1$ is the mean curvature of $\Sigma^{n-1}_1$ in $(\Omega^n,g)$ with respect to inward unit normal vector, and $H_2$ is the mean curvature of $\Sigma^{n-1}_2$ in $(\Omega^n,g)$ with respect to outward unit normal vector, we say it is a trivial NNSC-cobordism if each $\Sigma_i^{n-1}$, $i=1,2$, is differmorphic to $\Sigma^{n-1}$ and $\Omega^n$ is differmorphic to $\Sigma^{n-1}\times [0,1]$.} By the obvious way, we can define $(\Sigma_i^{n-1}, \gamma_i, H_i)$, $i=1,2$, to be cobordant with scalar curvature with a lower bound.

\begin{tikzpicture}

\draw[dashed] (2,1) arc (0:180:2cm and 0.8cm);
\draw (2,1) arc (0:-180:2cm and 0.8cm);
\node[below] at (0,1) {$(\Sigma^{n-1}_1, \gamma_1, H_1)$};
\draw (0,6) ellipse (3cm and 1.2cm)node[above]{$(\Sigma^{n-1}_2, \gamma_2, H_2)$};
\coordinate (A) at (-3,6);
\coordinate (B) at (3,6);
\coordinate (C) at (-2,1);
\coordinate (D) at (2,1);
\coordinate (E) at (-2.2,1);
\coordinate (F) at (-2.3,1.5);
\draw (C)--(A);
\draw (D)--(B);
\draw[->] (E)--node[left]{normal vector $\nu$}(F);
\node[below] at (0,3) {$(\Omega^n, g)$};
\node[below] at (0,-0.3) {\textbf{Figure 1}};
\node[below] at (0,-0.8) {$(\Omega^n, g)$ is a cobordism of {$(\Sigma^{n-1}_1, \gamma_1, H_1)$} and {$(\Sigma^{n-1}_2, \gamma_2, H_2)$}.};
\end{tikzpicture}

   NNSC-cobordism of Bartnik data may have relation with positive scalar curvature (PSC) concordant relation for two PSC-metrics on a manifold. Indeed, PSC-concordance is a special case of  trivial NNSC-cobordism of Bartnik data. For the definition and deep discussion on this topic  from topological  point of view,  please see \cite{Wa2, Wa3} and references therein. One basic problem in Riemannian Geometry  is to study NNSC fill-ins of Bartnik data $(\Sigma^{n-1}, \gamma, H)$, i.e. {\it under what  conditions is it that  $\gamma$ is  induced  by a Riemannian metric $g$ with  nonnegative  scalar curvature, for example, defined on $\Omega^n$, and $H$ is the mean curvature of $\Sigma^{n-1}$ in $(\Omega^n, g)$ with respect to the outward unit normal vector? } Indeed, this problem has been proposed by M. Gromov recently (see Problem A in \cite{Gromov2}, and  section 3.3 and  section 3.6 in \cite{Gromov4}). We note that, according to the definition of NNSC-cobordism, if $(\Sigma_1^{n-1}, \gamma_1, H_1)$ admits a PSC fill-in and $(\Sigma^{n-1}_1, \gamma_1, H_1)$ is NNSC-cobordant to $(\Sigma^{n-1}_2, \gamma_2, H_2)$ then $(\Sigma^{n-1}_2, \gamma_2, H_2)$ also  admits a PSC fill-in. Hence, we see that  NNSC-cobordism also has close relationship with NNSC fill-ins problem. Unfortunately, to our knowledge, there is no very effective criteria for NNSC-cobordism of Bartnik data. In this paper, we will present some results in this direction. By using  Hawking mass and Brown-York mass, we are able to  to give a criterion for non existence of  trivial NNSC-cobordism  when $n=3$.  For the definition of Hawking mass and Brown-York mass, see Definition \ref{hawkingmass} and Definition \ref{bymass} respectively. Namely, we have

\begin{theorem}\label{cobordism20}\quad
Suppose Bartnik data $(S^2, \gamma_i, H_i)$ satisfies $H_i>0$, $i=1,2$, and the Gaussian curvature of $(S^2, \gamma_2)$ is positive, then they admit no trivial NNSC cobordism provided
\begin{equation}\label{qmass1}
m_{BY}(S^2, \gamma_2, H_2)<0\le m_H(S^2, \gamma_1, H_1).	
\end{equation}

\end{theorem}

Note that assumption (\ref{qmass1}) means that the total mean curvature of Bartnik data $(S^2, \gamma_2, H_2)$ is not too small, and we find that if $(S^2, \gamma_1, H_1)$ is  NNSC-cobordant to  $(S^2, \gamma_2, H_2)$, this total mean curvature can't be too large. More precisely, we have

\begin{theorem}\label{2-d0}\quad
Suppose Bartnik data $(S^2, \gamma_i, H_i)$ satisfies $H_i>0, i=1,2$. Then there is a  positive constant $\Lambda(\gamma_1,\gamma_2,H_1)$  depending only on $\gamma_1,\gamma_2,H_1$ such that if $(S^2, \gamma_1, H_1)$ is NNSC-cobordant to  $(S^2, \gamma_2, H_2)$, then
$$\int_{S^2} H_2 d\mu_{\gamma_2}\le \Lambda.$$
\end{theorem}

For higher dimensional case, we have
\begin{theorem}\label{highdimnoncob0}\quad
Suppose Bartnik data $(S^{n-1}, \gamma_i, H_i)$ satisfies $H_i>0$, $i=1,2$, $3\leq n \leq 7$. Then there exists  a positive constant $\Lambda(n,\gamma_1,\gamma_2,H_1)$ depending only on $n,\gamma_1,\gamma_2, H_1$, such that for any $H_2\ge \Lambda$,  $(S^{n-1}, \gamma_1, H_1)$ is not NNSC-cobordant to $(S^{n-1}, \gamma_2, H_2)$.
\end{theorem}

\begin{remark}
Note on the other hand, with the same conditions in Theorem \ref{highdimnoncob0}, we have an existence result, i.e,  there exists $\Lambda(n,\gamma_1,\gamma_2,H_2)>0$, such that for any $H_1\ge \Lambda$,  $(S^{n-1}, \gamma_1, H_1)$ is NNSC-cobordant to $(S^{n-1}, \gamma_2, H_2)$. One can follow the arguments in Proposition \ref{extend} to construct an NNSC-cobordism of $(S^{n-1}, \gamma_1, H_1)$ and
$(S^{n-1}, \gamma_2, H_2)$.
\end{remark}

Inspired by the recent work \cite{SWW} and \cite{M2}, we are able to show
\begin{theorem}\label{highdimnoncob10}\quad
Suppose $3\leq n \leq 7$, given Bartnik data $(\Sigma_1^{n-1}, \gamma_1, 0)$ and $(\Sigma_2^{n-1}, \gamma_2, H_2)$ with $H_2>0$, and $(\Sigma_2^{n-1}, \gamma_2)$ is the boundary of some n-dimensional compact manifold $\Omega_2^n$, then there exists  a  positive constant $\Lambda(n,\gamma_2)$ depending only on $n, \gamma_2$, such that for any $H_2\ge \Lambda$, $(\Sigma_1^{n-1}, \gamma_1, 0)$ and $(\Sigma_2^{n-1}, \gamma_2, H_2)$ admit no NNSC-cobordism.
\end{theorem}

It is interesting to see that the constant $\Lambda$ in Theorem \ref{highdimnoncob10} is independent on  $(\Sigma_1^{n-1}, \gamma_1)$.

The rest of the paper is organized as follows. In section 2, we introduce some basic concepts, some lemmas for smoothing corners along some hypersurfaces in the manifold and deforming the scalar curvature of the manifold and the mean curvature of it's boundary. Moreover, we show that if $(S^{n-1}, \gamma_1, H_1)$ is NNSC-cobordant to $(S^{n-1}, \gamma_2, H_2)$, then the cobordism can be extended to a cobordism of two round spheres with constant mean curvature, and the scalar curvature of the extended cobordism has a lower bound depending on the initial Bartnik data. In section 3, we give the proofs for Theorem 1.1-1.4.

The authors would like to thank the referees for their careful work and constructive suggestions.

\section{Some Preliminary Lemmas}
\subsection{Some basic notions}
Let $M$ be an oriented $n$-dimensional smooth differentiable manifold with nonempty boundary $\p M$, $\Sigma$ is a smooth hypersurface in $M$ and $M\setminus \Sigma=M_-\cup M_+$.
Let's first review some relevant notions.
\begin{definition}(cf. Definition 1 in \cite{M})\quad
A metric ${\cal{G}}$ admitting corners along $\Sigma$ is defined to be a pair of $(g_-,g_+)$, where $g_-$ and $g_+$ are $C_{loc}^{2,\alpha}$ metrics on $M_-$ and $M_+$ so they are $C^2$ up to the boundary and they induce the same metric on $\Sigma$.
\end{definition}

\begin{definition}\quad
Given ${\cal{G}}=(g_-,g_+)$ on $M$, ${\cal{G}}$ is called asymptotically flat if
$(M_+, g_+)$ is asymptotically flat (AF) in the usual sense, i.e,
there is a compact subset $K$ containing $M_-$ such that, $M\setminus K$ is diffeomorphic to $\mathbb{R}^n\setminus\bar B^n_1(0)$, and in this coordinates chat, $g_+$ satisfies
\begin{equation}
g_{+ij}=\delta_{ij}+O(|x|^{-\tau}),~~\p_k g_{+ij}=O(|x|^{-\tau-1}), ~~\p_k\p_l g_{+ij}=O(|x|^{-\tau-2})
\end{equation}
for some $\tau>\frac{n-2}{2}$, where the partial derivatives are taken with respect to the Euclidean metric  and the Cartesian coordinates on $M\setminus K$.
\end{definition}

When the scalar curvature $R_{g_+}$ is $L^1$-integrable on $M_+$, the ADM mass can be defined on $M_+$.
\begin{definition}\quad
The Arnowitt-Deser-Misner (ADM) mass of $(M_+, g_+)$ is defined as:
\begin{equation}
m_{ADM}=\lim_{r\rightarrow\infty}\dfrac{\omega_{n-1}}{2(n-1)}\int_{S_r}(\p_i g_{+ij}-\p_j g_{+ii})\nu^jd\mu_r
\end{equation}
where $S_r$ is the coordinate sphere near the infinity,  $\nu$ is the outward unit normal vector.
\end{definition}

We also recall the following two kinds of quasi-local masses for surfaces in 3-dimensional manifolds.
\begin{definition}\label{hawkingmass}\quad
The Hawking mass of a surface $\Sigma$ in $M^3$ is defined as:
\begin{equation}
m_H(\Sigma)=\dfrac{|\Sigma|^{1/2}}{(16\pi)^{1/2}}(1-\dfrac{1}{16\pi}\int_{\Sigma} H^2 d\mu).
\end{equation}
Here $|\Sigma|$ denotes the area of $\Sigma$.
\end{definition}

\begin{definition}\label{bymass}\quad
The Brown-York mass of a topological sphere  $(\Sigma^2, \gamma)$ in $M^3$ with positive Gauss curvature is defined as:
\begin{equation}
m_{BY}(\Sigma)=\dfrac{1}{8\pi}\int_{\Sigma} (H_0-H)d\mu
\end{equation}
where $H_0$ is the mean curvature  of the image of isometric embedding $(\Sigma^2, \gamma)$  into $\mathbb{R}^3$, here and in the sequel our mean curvature is always with respect to the outward unit normal vector.
\end{definition}

Similarly, the asymptotically hyperbolic manifold with corners is defined as follows:
\begin{definition}\quad
Given ${\cal{G}}=(g_-,g_+)$ on $M$, ${\cal{G}}$ is called asymptotically hyperbolic if
\begin{enumerate}
  \item $M$ is conformally compact and the conformal infinity $\p \bar M$ is the standard sphere.
  \item there is a unique defining function $r$ in a collar neighborhood of $\p \bar M$ such that
  $$g_+=\sinh^{-2} r(dr^2+\gamma_{std}+\dfrac{h}{n}r^n+O(r^{n+1})),$$
  \item the asymptotic expansion can be differentiated twice.
\end{enumerate}
$Tr_{\gamma_{std}} h$ is called the mass aspect of $(M,{\cal{G}})$  and $\int_{S^{n-1}} Tr_{\gamma_{std}} hd\mu_{\gamma_{std}}$ is called the mass of $(M,{\cal{G}})$.
\end{definition}

\subsection{Some estimates for conformal deformation equations}
In \cite{M}, the author gives a way to smooth the corners along $\Sigma$. we define $\tilde M$ to be a possibly new differentiable manifold with the background topological space $M$, and then ${\cal{G}}$ becomes a continuous metric $g$ on $\tilde M$ since $g_-|_{\Sigma}=g_+|_{\Sigma}$. We list some results in \cite{M} here for the reader's convenience.
\begin{proposition}(cf. Proposition 3.1 in \cite{M})\label{Smoothing1}\quad
Given ${\cal{G}} = (g_-, g_+)$, if $H(\Sigma, g_-)\ge H(\Sigma, g_+)$, then there is a
family of $C^2$ metrics $\{g_\delta\}_{0<\delta \le \delta_0}$ on $\tilde M$ such that $g_\delta$ is uniformly close to $g$ on
$\tilde M$,
\begin{align}
&R_\delta(x,t)=O(1), ~~\text{for}~~ (x,t)\in  \Sigma\times \{-\frac{\delta^2}{100}|t|\le \frac{\delta}{2}\}\\
&R_\delta(x,t)=O(1)+(H(\Sigma, g_-)-H(\Sigma, g_+))\{\frac{\delta^2}{100}\varphi(\frac{\delta^2}{100})\}, ~~\text{for}~~ (x,t)\in  \Sigma\times [-\frac{\delta^2}{100},\frac{\delta^2}{100}]
\end{align}

where $O(1)$ represents quantities that are bounded by constants depending
only on ${\cal{G}}$, but not on $\delta$, and $\varphi\in C_c^\infty[-1,1]$ is a standard mollifier with $0\le\varphi\le 1$, $\varphi\equiv 1$ on $[-\frac{1}{3},\frac{1}{3}]$, and $\int_{-1}^1 \varphi=1$.
\end{proposition}

When $H(\Sigma, g_-)\equiv H(\Sigma, g_+)$, we have
\begin{corollary}(cf. Corollary 3.1 in \cite{M})\label{Smoothing2}\quad
Given ${\cal{G}} = (g_-, g_+)$, if $H(\Sigma_0, g_-)\equiv H(\Sigma_0, g_+)$, then there is a
family of $C^2$ metrics $\{g_\delta\}_{0<\delta \le \delta_0}$ on $\tilde M$ so that $g_\delta$ is uniformly close to $g$ on
$\tilde M$, $g_\delta = {\cal{G}}$ outside $\Sigma\times(- \frac{\delta}{2},\frac{\delta}{2})$ and the scalar curvature of $g_\delta$ is uniformly
bounded inside $\Sigma\times [-\frac{\delta}{2},\frac{\delta}{2}]$ with bounds depending only on ${\cal{G}}$, but not on $\delta$.
\end{corollary}

In the following, we still use $M$ instead of $\tilde M$ to denote the manifold after smoothing corners. We focus on the case that both $R_{g_-}$ and $R_{g_+}$ are nonnegative  and there is a point $p$ in $M$ and a constant $\delta_0$ with $R_{g_-}(p)\geq \delta_0>0$  or $R_{g_+}(p)\geq \delta_0>0$,  and then we can find a  smooth and  nonnegative scalar curvature metric $g$ on $M$ such that $g|_{\p M}=g_-|_{\p M}$ (or $g|_{\p M}=g_+|_{\p M}$) and $H_g(\p M)=H_{g_-}(\p M)$ (or $H_g(\p M)=H_{g_+}(\p M)$)  by performing twice conformal deformations. We deal with the following cases respectively:
\begin{enumerate}
  \item $M$ is an asymptotically flat manifold with an inner boundary $\p M$;
  \item $M$ is a compact manifold with boundary $\p M$($\p M$ may have more than one connected component);
\end{enumerate}

For case 1, consider the following equation
\begin{equation}\label{conformal1}
\left\{
\begin{split}
&\Delta u_\delta+\dfrac{1}{c(n)}R_{g_{\delta-}}u_\delta = 0, ~~\text{in}~~ M\\
&u_\delta = 1, ~~\text{on} ~~\partial M\\
&\lim_{|x|\rightarrow\infty}u_\delta(x) = 1
\end{split}
\right.
\end{equation}
where $R_{g_{\delta-}}$ is the negative part of $R_{g_{\delta}}$, i.e., $R_{g_{\delta-}}=\max\{0, -R_{g_{\delta}}\}\geq 0$.

The solvability of \eqref{conformal1} is guaranteed by the following:
\begin{lemma}(cf. Lemma 3.1 in \cite{SWY}, see also \cite{M} and \cite{SchY})\label{CD1}\quad
Let $(M,g_M)$ be an asymptotically flat n-manifold with inner boundary $\p M$, and $h$
be a function that has the same decay rate at $\infty$ as $R_{g_M}$, then there exists a number $\e_N>0$
depending only on the $C^0$ norm of $g_M$ and the decay rate of $g_M$, $\p g_M$ and $\p^2 g_M$ at $\infty$ so that if
\begin{equation*}
\left(\int_M |h_-|^{\frac{n}{2}}d\mu_{g_M}\right)^\frac{2}{n}\le \e_N
\end{equation*}
where $h_-$ is the negative part of $h$, then
\begin{equation}
\left\{
\begin{split}
\Delta u-hu&=0, ~~\text{in}~~ M\\
u&=1, ~~\text{on} ~~\p M\\
\lim_{|x|\rightarrow\infty}u&=1,
\end{split}
\right.
\end{equation}
has a $C^2$ positive solution $u$ such that
\begin{equation}\label{u asymp}
u=1+\dfrac{A}{|x|^{n-2}}+B
\end{equation}
for some constant $A$ and some function $B$, where $B=O(|x|^{1-n})$ and $\p B=O(|x|^{-n})$, here $\p B$ denote the partial derivative.
\end{lemma}

For case 2, consider the following equation
\begin{equation}\label{conformal2}
\left\{
\begin{split}
&\Delta u_\delta+\dfrac{1}{c(n)}R_{g_{\delta-}}u_\delta =0, ~~\text{in}~~ M\\
&u_\delta =1, ~~\text{on } ~~\p M
\end{split}
\right.
\end{equation}
where $c(n)=\frac{4(n-1)}{n-2}$.

Similarly, we establish the following lemma to guarantee the solvability of \eqref{conformal2}
\begin{lemma}\label{CD2}\quad
Let $(M,g_M)$ be a compact n-manifold with boundary $\p M$, then there exists a number $\e_N>0$
depending only on the $C^0$ norm of $g_M$ so that if the function $h$ satisfies
\begin{equation*}
\left(\int_M |h_-|^{\frac{n}{2}}d\mu_{g_M}\right)^\frac{2}{n}\le \e_N
\end{equation*}
where $h_-$ is the negative part of $h$, then
\begin{equation}
\left\{
\begin{split}
&\Delta u-hu=0, ~~\text{in}~~ M\\
&u=1, ~~\text{on} ~~\p M
\end{split}
\right.
\end{equation}
has a unique $C^2$ positive solution $u$.
\end{lemma}
\begin{proof}\quad
The proof is similar to Lemma 3.2 in \cite{SchY}, even more easier, since we don't need to consider the asymptotic property at the infinity, we omit it here.
\end{proof}

Moreover, for both cases, the maximum principle implies that $u_\delta\ge 1$. By a similar argument to Proposition 4.1 and Lemma 4.2 in \cite{M}, we have
\begin{lemma}\label{CT}\quad
 Let $u_\delta$ be the solution of \eqref{conformal1} or \eqref{conformal2}  then it satisfies
$$\lim\limits_{\delta\rightarrow 0} ||u_\delta-1||_{L^\infty(M)}=0$$
and $||u_\delta||_{C^{2,\alpha}}\le C_K$, Here $K$ is any compact set in $M\setminus \Sigma$, $C_K$ is a constant depending only on ${\cal{G}}$ and $K$.
Let $\tilde g_\delta=u_\delta^{\frac{4}{n-2}} g_\delta$ be the resulting metrics, the scalar curvature of $\tilde g_\delta$ is nonnegative.
Moreover, for the case 1 the mass of $\tilde g_\delta$ converges to the mass of ${\cal{G}}$ as $\delta\rightarrow 0$.
\end{lemma}

\begin{remark}\label{Interpolation}
By $\lim\limits_{\delta\rightarrow 0} ||u_\delta-1||_{L^\infty(M)}=0$, $||u_\delta||_{C^{2,\alpha}}\le C_K$
and the interpolation inequality(see Lemma 6.35 in \cite{GT}), we can see that $\lim\limits_{\delta\rightarrow 0} |\nabla u_\delta|=0$.
\end{remark}

After the conformal deformation, the mean curvature of the boundary is deformed by
$$H_{\tilde g_\delta}=(\dfrac{c(n)}{2}\dfrac{\p u_\delta}{\p \nu}+H_{g_\delta}u_\delta)u_\delta^{-\frac{n}{n-2}}=\dfrac{c(n)}{2}\dfrac{\p u_\delta}{\p \nu}+H_{g_\delta}.$$
here $\nu$ denote the unit outward normal vector, by the maximum principle we have
$$\frac{\p u_\delta}{\p \nu}\leq 0,$$
then
$$H_{\tilde g_\delta}\leq H_{g_\delta}=H_g.$$

Now we preform another conformal deformation, i.e., solve the following
\begin{equation}\label{conformal with eta}
\left\{
\begin{split}
&\Delta w-\dfrac{1}{c(n)}\eta R_{{tilde g_\delta}}w =0, ~~\text{in}~~ M\\
&w =1, ~~\text{on} ~~\partial M
\end{split}
\right.
\end{equation}
where $\eta$ is a cut-off function on a neighborhood $U_p$ of $p$ in M  where $R_{\tilde g_\delta}$, we get a solution w.

Let $\hat g_\delta=w^{\frac{4}{n-2}} \tilde g_\delta$,  the scalar curvature
\begin{equation*}
R_{\hat g_\delta}=w^{-\frac{n+2}{n-2}}(-{c(n)}\Delta w+R_{\tilde g_{\delta}}w)=w^{-\frac{4}{n-2}}((1-\eta) R_{\tilde g_{\delta}})\ge 0
\end{equation*}
and the mean curvature of the boundary is
\begin{equation}
H_{\hat g_\delta}=(\dfrac{c(n)}{2}\dfrac{\p w}{\p \nu}+H_{\tilde g_\delta}w)w^{-\frac{n}{n-2}}=\dfrac{c(n)}{2}\dfrac{\p w}{\p \nu}+H_{\tilde g_\delta}=\dfrac{c(n)}{2}(\dfrac{\p u_\delta}{\p \nu}+\dfrac{\p w}{\p \nu})+H_{g_\delta}.
\end{equation}
here $\nu$ is still the unit outward normal vector since conformal deformation do not change the direction of the normal vector, and by maximum principle, we have $\frac{\p w}{\p \nu}>0$.
Note by lemma \ref{CD2}, we have $\lim\limits_{\delta\rightarrow 0}u_\delta=1$, which implies that
$$\lim\limits_{\delta\rightarrow 0}\frac{\p u_\delta}{\p \nu}=0,$$
 then take $\delta$ sufficiently small, we have
 $$(\frac{\p u_\delta}{\p \nu}+\frac{\p w}{\p \nu})>0, $$i.e.,
 $$H_{\hat g_\delta}>H_{g_\delta}.$$

\begin{remark}
When both $R_{g_+}\ge -C_+$ and $R_{g_-}\ge -C_-$, we can get a  smooth metric $g$ with $R_g\ge -{c(n)}(C+1)$ on the whole manifold in a similar way, where ${C=\max\{C_+,C_-\}}$. To see this, solve the following equation
\begin{equation}\label{conformal3}
\left\{
\begin{split}
&\Delta u_\delta+\dfrac{1}{c(n)}(R_{g_{\delta}}+C)_-u_\delta =0, ~~\text{in}~~ M\\
&u_\delta =1, ~~\text{on } ~~\p M
\end{split}
\right.
\end{equation}
where $(R_{g_{\delta}}+C)_-$ denote the negative part of $R_{g_{\delta}}+C$. The \eqref{conformal3} is also solvable since $(R_{g_{\delta}}+C)_-$ satisfies the condition in Lemma \ref{CD2}. Then let $\tilde g_\delta=u_\delta^{\frac{4}{n-2}} g_\delta$,
Then $$R_{\tilde g_\delta}=u_\delta^{\frac{n+2}{n-2}}(R_{g_\delta}u_\delta-{{c(n)}}\Delta u_\delta)={{c(n)}}u_\delta^{\frac{4}{n-2}}((R_{g_\delta}+C)_+ -C)\ge -{{c(n)}}C$$
since $u_\delta\ge 1$ by the maximum principle. $u_\delta$ satisfies the same property in Lemma \ref{CT}. Also after another conformal deformation as above the mean curvature $H_{\hat g_\delta}>H_{g_\delta}$.
\end{remark}

The following lemma is very useful in some gluing construction. In fact, it has been established in the proof of Theorem 1.5 in \cite{SWWZ}, we include a proof here for the sake of completeness.
\begin{proposition}\label{DBH2}\quad
 Let $(\Omega, g)$ be a compact manifold with boundary $\p \Omega=\Sigma$, $g|_{\p \Omega}=\gamma$, $H$ is a smooth function on $\Sigma$.
 The mean curvature $H_g$ satisfies $H_g>H$ and the scalar curvature $R_g$ of $(\Omega, g)$ satisfies one of the following condition:
 \begin{enumerate}
   \item $R_g\ge S$, $R_g>S$ at some point $p\in \Omega$, where $S$ is some nonnegative constant,
   \item $R_g>S$, where $S$ is a constant.
 \end{enumerate}
Then we can find a metric $\tilde g$ on $\Omega$ such that $R_{\tilde g}>S$, $\tilde g|_{\p \Omega}=\gamma$ and $H_{\tilde g}=H$.
\end{proposition}
\begin{proof}\quad
The first case can be reduced to the second one by the following perturbations.
We first consider the case $S=0$, by the assumption, $R_g>0$ in some neighborhood $U_p$ of $p$. Let $\eta$ be a cut-off function on $U_p$.
We'll increase the lower bound of the scalar curvature a little by a conformal deformation. For $\epsilon\ge 0$, consider the following equation
\begin{equation}\label{conformal with epsilon}
\left\{
\begin{split}
&\Delta w_\epsilon-\dfrac{1}{c(n)}\eta R_{g}w_{\e} =-\epsilon, ~~\text{in}~~ \Omega\\
&w_\e =1, ~~\text{on} ~~\Sigma
\end{split}
\right.
\end{equation}
For sufficiently small $\epsilon$,  equation \eqref{conformal with epsilon} has a positive and smooth solution $w_\epsilon$.

Let $g_1=w_\e^{\frac{4}{n-2}} g$, then
\begin{equation*}
R_{ g_1}=w_\e^{-\frac{n+2}{n-2}}(-c(n)\Delta w_\e+R_{g}w_\e)=w_\e^{-\frac{4}{n-2}}((1-\eta) R_{g}+c(n)\e w_\e^{-1} )\ge c(n)\e w_\e^{-\frac{n+2}{n-2}}>0.
\end{equation*}
On $\Sigma$ we have
$$H_{g_1}=\dfrac{c(n)}{2}\dfrac{\p w_\e}{\p \nu}+H_{ g}.$$
When $\e=0$, similar to equation \eqref{conformal with eta}, we have $\frac{\p w_0}{\p \nu}|_\Sigma>0$. By $C^2$ compactness of solutions of \eqref{conformal with epsilon}, we obtain $$H_{g}+\dfrac{c(n)}{2}\frac{\p w_\e}{\p \nu}>H$$
for $\e$ sufficiently small, and then $H_{g_1}>H$.

For $S>0$,  the scalar curvature can also be perturbed in the similar way, see more details in the proof of Theorem 1.5 in \cite{SWWZ}.

Now we can assume $R_g>S$ and we'll construct a metric on a small collar neighborhood of $\Sigma$. For some small $t_0> 0$, $\Sigma\times [-t_0, 0]$ is diffeomorphic to a $t_0$-collar neighborhood of $\Sigma$ in $\Omega$. Let $\Sigma_t$ denote $\Sigma\times \{t\}$ and identify $\Sigma\times \{0\}$ with $\Sigma$. In this $t_0$-collar neighborhood, $g_1$ can be written by
$g_1=dt^2+\bar g_1(t)$, where $\bar g_1(t)$ is the induced metric on $\Sigma_t$ from $g_1$.
Let $m:\Sigma\rightarrow \mathbb{R}$ be a positive smooth function such that
\begin{equation*}
m(x)=\dfrac{H_{g_1}(x)-H(x)}{n-1}.
\end{equation*}
Let $\lam$ be a smooth function on $[-t_0,0]$ such that $\lam(0)=0$, $\lam'(0)=-1$. Let $g_2$ be a metric on $\Sigma\times [-t_0, 0]$ with
$$g_2(x,t)=dt^2+(1+m(x)\lam(t))^2\bar g_1(x,t).$$
and extend $g_2$ on the whole $\Omega$. Then we immediately obtain $g_2|_{\Sigma}=\gamma$  and $\frac{d}{dt}|_{\Sigma}$ is the outer normal vector. Let $\bar g_2(x,t)=(1+m(x)\lam(t))^2\bar g_1(x,t)$, then it is the induced metric on $\Sigma_t$ from $G_2$.

Let $A_i(t)$ and $H_i(t)$ denote the second fundamental form and the mean curvature of $\Sigma_t$ with respect to $g_i$ respectively, $i=1,2$. Then we have
$$A_2=(1+m\lam)^2A_1+m\lam'(1+m\lam)\bar g_1,$$
and
$$H_2=H_1+\dfrac{(n-1)m\lam'}{1+m\lam}.$$
Moreover, $H_{g_2}=H_2(0)=H$.

Let $\bar R_i(t)$ denote the scalar curvature of $\Sigma_t$ with respect to $\bar g_i$, $i=1,2$.
By a simply computation we have
$$\bar R_2=(1+m\lam)^{-2}\left(\bar R_1-\dfrac{2(n-1)m\lam}{1+m\lam}\Delta_{\bar g_2}m-\dfrac{(n-1)(n-4)\lam^2}{(1+m\lam)^2}|\nabla_{\bar g_2}m|^2\right).$$

By Gauss equation,
$$R_{g_i}=\bar R_i-2\dfrac{\p H_i}{\p t}-|H_i|^2-|A_i|^2~~(i=1,2).$$
Then
\begin{equation*}
\begin{split}
R_{g_2}=&R_{g_1}-2\dfrac{\p (H_2-H_1)}{\p t}+(\bar R_2-\bar R_1)-(|H_2|^2-|H_1|^2)-(|A_2|^2-|A_1|^2)\\
=&R_{g_1}-\dfrac{2(n-1)m\lam''}{1+m\lam}-\dfrac{(n-1)(n-2)m^2\lam'^2}{(1+m\lam)^2}-\dfrac{2(n-1)m\lam}{(1+m\lam)^3}\Delta_{\bar g_2}m\\
&-\dfrac{(n-1)(n-4)\lam^2}{(1+m\lam)^4}|\nabla_{\bar g_2}m|^2-\dfrac{m\lam(2+m\lam)}{(1+m\lam)^2}\bar R_{g_1}-\dfrac{2nm\lam'}{1+m\lam}H_1.
\end{split}
\end{equation*}

Note that $H_1$, $\bar R_1$, $\Delta_{\bar g_2}m$ and $|\nabla_{\bar g_2}m|^2$ are bounded in $[-t_0,0]$. Choose $\lam''(t)\ll-1$ in small interval $[-t_1,0]$ around $t=0$, then $R_{g_2}>0$ in this small interval, here $t_1<t_0$.

By the perturbation and above construction, there exist some $\epsilon_1$ such that $R_{g_1}\ge S+\epsilon_1$ in $\Omega$, and \\
$R_{g_2}\ge S+\epsilon_1$ in $\Sigma\times [-t_1, 0]$. Then apply Theorem 5 in \cite{BMN}, we can get a new metric $\tilde g$, with $R_{\tilde g}>S$, $\tilde g|_{\p \Omega}=\gamma$ and $H_{\tilde g}=H$, which is the desire metric.
\end{proof}

By Lemma \ref{DBH2}, we immediately have
\begin{corollary}
If $(S^{n-1}, \gamma_1, H_1)$ is  NNSC-cobordant to $(S^{n-1}, \gamma_2, H_2)$, then $(S^{n-1}, \gamma_1, H_1)$ is  NNSC-cobordant to $(S^{n-1}, \gamma_2, H)$ for any $H<H_2$.
\end{corollary}

Moreover, joint with the above twice conformal deformations, we have
\begin{corollary}\quad
Let ${\cal{G}} = (g_-, g_+)$ be a metric with corners along $\Sigma$ on $M$. If $H(\Sigma, g_-)\ge H(\Sigma, g_+)$, both $R_{g_-}$ and $R_{g_+}$ are nonnegative, but they don't vanish at the same time. Then we can find a smooth nonnegative scalar curvature metric $g$ on $M$ such that $g|_{\p M}=g_-|_{\p M}$ and $H_g(\p M)=H_{g_-}(\p M)$.
\end{corollary}

\subsection{Extension of cobordism}
 The following extending proposition implies that once we have an NNSC-cobordism of $(S^{n-1}, \gamma_i, H_i)$, $i=1,2$, we will have a cobordism of round spheres with scalar bounded below. This is very useful in the proof of nonexistence of NNSC-cobordism for some Bartnik data.

 \begin{proposition}\label{extend}\quad
Suppose Bartnik data $(S^{n-1}, \gamma_i, H_i)$ satisfies $H_i>0, i=1,2$,  $(S^{n-1}, \gamma_1, H_1)$ is NNSC-cobordant to $(S^{n-1}, \gamma_2, H_2)$ and their cobordism denoted by  $(\Omega_0, g_0)$, then we can extend it to a cobordism $(\Omega_1, g_1)$  of $(S^{n-1}, \tilde\gamma_1, \tilde H_1)$ and $(S^{n-1}, \tilde\gamma_2, \tilde H_2)$,  where $\tilde\gamma_1$ and $\tilde\gamma_2$ are round metrics, $\tilde H_1$ and $\tilde H_2$ are positive constants, and
\begin{description}
\item[(1)] $(S^{n-1}, \tilde\gamma_1, \tilde H_1)$ is cobordant to $(S^{n-1}, \tilde\gamma_2, \tilde H_2)$;
\item[(2)]  The scalar curvature $R_{g_1}\geq -C$,   $C$ is a constant depends only on $\gamma_1$, $\gamma_2$;
\item[(3)] For any fixed $(S^{n-1}, \tilde\gamma_2)$,  $\tilde H_2$ can be arbitrarily large provided $H_2$ is large enough;
\item[(4)] $\tilde H_i$ depends only on $H_i$, $\gamma_i$ and $\tilde \gamma_i$, $i=1,2$.
\item[(5)] $\Omega_0\subset \Omega_1$;
\item[(6)] When both $\gamma_1$ and $\gamma_2$ are isotopic to the standard sphere metric $\gamma_{std}$ in
$$\mathcal{M}_{psc}(S^{n-1})=\left\{\gamma|\gamma ~~\text{is a smooth metric on} ~~S^{n-1} ~~\text{and}~~ R_\gamma>0\right\},$$
 then the scalar curvature $R_{g_1}\geq 0$.
\end{description}
\end{proposition}

\begin{remark}\quad
Note that  $\mathcal{M}_{psc}(S^{n-1})$ is path-connected when $n=3,4$(see \cite{Ma}), and in this case we can always get an extension of the given NNSC-cobordism with nonnegative scalar curvature.
\end{remark}

\begin{proof}[Proof of Proposition \ref{extend}]\quad
Our arguments are inspired by those in  \cite{SWW}. The proof have mainly three steps:
\begin{description}
  \item[(1)] Extend the cobordism such that each boundary metric is round.
  \item[(2)] Smooth the corners and keep the boundary unchanged.
  \item[(3)] Deform the mean curvature of each boundary to be a constant.
\end{description}

We first show $(S^{n-1}, \gamma_2, H_2)$ can be cobordant to some $(S^{n-1}, \tilde \gamma_2, H'_2)$, where $\tilde \gamma_2$ is a round metric. As in \cite{SWW}, let $\gamma(t)=t\gamma_2+(1-t)\gamma_{std}$ for $t \in [0,1]$. Choose a constant $a_2>0$ large
enough so that $e^{2a_2t_2}\gamma(t_2)> e^{2a_2t_1}\gamma(t_1)$ for any $0\le t_1\le t_2\le 1$, and denote
$\bar \gamma(t)=e^{2a_2t}\gamma_{t}$. Let $K=\min_{S^{n-1}}R_{\bar\gamma(t)}$  which depends only on $\gamma_2$, $\bar A(t)$ and
$\bar H(t)$ denote the second fundamental form and the mean curvature of $\Sigma(t)$ induced from the metric $\bar g= dt^2 + \bar \gamma(t)$. Since $\bar \gamma(t)$ strictly monotonically increases, $\bar A(t)>0$,  i.e., $\bar \gamma(t_1)-\bar \gamma(t_2)$ is positive definite for $t_1>t_2$ and $\bar A(t)$ is positive definite. It follows that $\bar H(t)>0$ and $\bar H(t)^2-  |\bar A(t)|^2>0$.

Consider
the quasi-spherical metric equation
\begin{equation}\label{QS'}
\left\{
\begin{split}
\bar H(t)\dfrac{\p u}{\p t} &=u^2\Delta_{\bar \gamma(t)} u+\dfrac{1}{2}(u-u^3)R_{\bar \gamma(t)}-\dfrac{1}{2}u R_{\bar g}+\dfrac{1}{2}u^3 K\\
u(x,0) &=\dfrac{\bar H(0)}{H_2}>0,
\end{split}
\right.
\end{equation}
Note $R_{\bar \gamma(t)}-K\ge 0$, and by maximum principle, we see that the above equation \eqref{QS'} has solution on whole $[0,1]$, and $\|u\|_{C^0(S^{n-1}\times[0,1])}$ is very small if $H_2$ is large enough.  Then $$(D_2=S^{n-1} \times [0,1], dS^2_2=u^2dt^2+\bar \gamma)$$  gives a cobordism  of $(S^{n-1}, \gamma_2, H_2)$ and $(S^{n-1}, \gamma'_2, H'_2)$ with the scalar curvature having a lower bound depends only on $\gamma_2$, $H_2$, here $\gamma'_2=e^{2a_2}\gamma_{std}$ and $H'_2= u^{-1}(1) \bar H(1)$. Hence, $H'_2$ can be arbitrarily large provided $H_2$ is large enough.

By the similar arguments, we can show $(S^{n-1}, \tilde \gamma_1, H'_1)$ is NNSC-cobordant to $(S^{n-1}, \gamma_1, H_1)$, where $\tilde \gamma_1$ is a round metric and $H'_1$ is given by solving the relevant quasi-spherical metric equation, the cobordism is denoted by $(D_1, dS^2_1)$.

Now glue $(D_1, dS^2_1)$, $(\Omega_0, g_0)$ and $(D_2, dS^2_2)$ one by one and then smooth each corner along $(S^{n-1}, \gamma_i)$ and $(S^{n-1}, \gamma'_i)$ by Proposition \ref{Smoothing1} (or Corollary \ref{Smoothing2}), then perform the twice conformal deformation and  together with  Proposition \ref{DBH2}(note here we do not need a definitely lower bound of the scalar curvature,  Proposition \ref{DBH2} is still valid), we can get a cobordism $(\Omega_1, g_1)$ of $(S^{n-1}, \tilde\gamma_1,  H'_1)$ and $(S^{n-1}, \tilde\gamma_2,  H'_2)$ with $R_g\ge -C(\gamma_1,\gamma_2)$. Here $R_g$ depends only on $\gamma_1,\gamma_2$, since the scalar curvature of $dS^2_i$ depends only on $\gamma_i$ for each $i=1,2$ by it's construction.

Then using Proposition \ref{DBH2} again, we can get a cobordism of $(S^{n-1}, \tilde \gamma_2, \min H'_2)$ and $(S^{n-1}, \tilde\gamma_1, \max H'_1)$, let $\tilde H_2=\min H'_2$ and $\tilde H_1=\max H'_2$. For convenience, we still denote the cobordism by $(\Omega_1, g_1)$.

When  $\gamma_2$ is isotopic to $\gamma_{std}$ in $\mathcal{M}_{psc}(S^{n-1})$, by Proposition 2.1 in \cite{CM}, we can find a smooth path $\gamma(t):[0,1]\rightarrow \mathcal{M}_{psc}(S^{n-1})$ such that $\gamma|_{t=0}=\gamma_1$, $\gamma(t)=\gamma_{std}$ for $t\in [\frac{5}{6},1]$ . Let $\bar \gamma(t)= (1+t^2)\gamma(t)$, $\bar A(t)$ and $\bar H(t)$ denote the second fundamental form and the mean curvature of the slice $\Sigma_t$ induced from the metric $\bar g= dt^2 + \bar \gamma(t)$.  Note that $R_{\bar \gamma(t)}\geq 2\delta_0>0$  on $S^{n-1}$ for all $t\in [0,1]$. Then consider the following quasi-spherical equation:
\begin{equation}\label{QS''}
\left\{
\begin{split}
\bar H(t)\dfrac{\p u}{\p t} &=u^2\Delta_{\bar \gamma(t)} u+\dfrac{1}{2}(u-u^3)R_{\bar \gamma(t)}-\dfrac{1}{2}u R_{\bar g}{+\dfrac{1}{2}\delta_0 u^3 }\\
u(x,0) &=\dfrac{\bar H(0)}{H_2}>0,
\end{split}
\right.
\end{equation}
it has global solution on $[0,1]$ since $R_{\bar \gamma(t)}\geq 2 \delta_0>0$. This gives a cobordism $(D_2, dS^2_2)$ of $(S^{n-1}, \gamma_2, H_2)$ and $(S^{n-1}, \tilde\gamma_2, H'_2)$ with $dS^2_2=(1+t^2)\gamma(t)+u^2dt^2$ and $R_{dS^2_2}=\delta_0>0$, where $\tilde\gamma_2=2\gamma_{std}$ and $H'_2= u^{-1}(1) \bar H(1)$. Similarly, we may construct NNSC-cobordism of $(D_1, dS^2_1)$ of $(S^{n-1}, \tilde\gamma_1, H'_1)$ and $(S^{n-1}, \gamma_1, H_1)$. Let $\tilde H_2=\min H'_2$ and $\tilde H_1=\max H'_1$.
Then perform the gluing and the conformal deformation \eqref{conformal2}, by Corollary \ref{Smoothing2} and Lemma \ref{CT}, and together with  Proposition \ref{DBH2},  we get the desired NNSC-cobordism of $(S^{n-1}, \tilde\gamma_1, \tilde H_1)$ and $(S^{n-1}, \tilde\gamma_2, \tilde H_2)$,  then we finish the proof of  Proposition \ref{extend}.
\end{proof}

\begin{tikzpicture}
\draw[dashed] (2,0) arc (0:180:2cm and 0.5cm);
\draw (2,0) arc (0:-180:2cm and 0.5cm);
\node[right](O) at (2,0) {$(S^{n-1}, \gamma_1, H_1)$};
\draw[dashed] (3,3) arc (0:180:3cm and 0.75cm);
\draw (3,3) arc (0:-180:3cm and 0.75cm);
\node[right](P) at (3,3) {$(S^{n-1}, \gamma_2, H_2)$};

\coordinate (A) at (-3,3);
\coordinate (B) at (3,3);
\coordinate (C) at (-2,0);
\coordinate (D) at (2,0);
\coordinate (E) at (-2.2,0);
\coordinate (F) at (-2.3,0.3);
\coordinate (G) at (-1.5,-1.5);
\coordinate (H) at (1.5,-1.5);
\coordinate (I) at (-3.5,4.5);
\coordinate (J) at (3.5,4.5);
\coordinate (K) at (-3.7,5.1);
\coordinate (L) at (3.7,5.1);
\coordinate (M) at (-1.3,-2.1);
\coordinate (N) at (1.3,-2.1);

\draw (C)--(A);
\draw (D)--(B);
\draw (C)--(G);
\draw (D)--(H);
\draw (A)--(I);
\draw (B)--(J);

\draw[dashed] (1.5,-1.5) arc (0:180:1.5cm and 0.375cm);
\draw (1.5,-1.5) arc (0:-180:1.5cm and 0.375cm);
\draw (3.5,4.5) arc (0:180:3.5cm and 0.875cm);
\node[right] at (3.5,4.5) {$(S^{n-1}, \tilde\gamma_2, \tilde H_2)$};
\draw (3.5,4.5) arc (0:-180:3.5cm and 0.875cm);
\node[right] at (1.5,-1.5) {$(S^{n-1}, \tilde\gamma_1, \tilde H_1)$};
\draw[dashed] (1.5,-1.5) arc (0:180:1.5cm and 0.375cm);
\draw (1.5,-1.5) arc (0:-180:1.5cm and 0.375cm);

\draw[->] (E)--node[left]{normal vector $\nu$}(F);
\node[below] at (0,-2.5) {\textbf{Figure 2}};
\end{tikzpicture}

The following proposition gives a way to construct NNSC cobordism of Bartnik data with round sphere and constant mean curvature, which is an existence result.
\begin{proposition}\label{constants}\quad
Suppose Bartnik data $(S^{n-1}, \gamma_i, H_i)$ satisfies that each $\gamma_i$ is a round metric with volume $V_i$ and each $H_i>0$ is a constant, $i=1,2$. Then each $(S^{n-1}, \gamma_i, H_i)$ can be regarded as an inner boundary of a Schwarzschild manifold with ADM mass $m_i$, $i=1,2$. If $V_1<V_2$, $m_1\le m_2$, then $(S^{n-1}, \gamma_1, H_1)$ and $(S^{n-1}, \gamma_2, H_2)$ admit an NNSC cobordism.
\begin{proof}\quad
Let $$r_i=\left(\dfrac{V_i}{\omega_{n-1}}\right)^{\frac{1}{n-1}},$$
where $\omega_{n-1}$ is the volume of the standard metric on $S^{n-1}$.

Let $g$ be a metric on $\Omega=S^{n-1}\times [r_1,r_2]$ such that
$$g=(1-\dfrac{2m(r)}{r^{n-2}})^{-1}dr^2+r^2\gamma_{std}$$
with $m(r_1)=m_1, m(r_2)=m_2$, and $m'(r)\ge 0$ on $[r_1, r_2]$.
Then we have $$g|_{S^{n-1}\times\{r_i\}}=\gamma_i,~~ H_g(S^{n-1}\times\{r_i\})=H_i,~~ i=1,2$$
and
$$R_g=\dfrac{2(n-1)m'(r)}{r^{n-1}}\ge 0.$$ This shows that $(\Omega, g)$ is the desire cobordism.
Moreover, when $m_1= m_2$, $(\Omega, g)$ is a part of a Schwarzschild manifold with ADM mass $m_1$.
 Note in this proof,  the ADM mass $m_1$ and $m_2$ of the Schwarzschild manifolds may be negative if the mean curvature $H_1$ and $H_2$ are large enough.
\end{proof}
\end{proposition}

\section{Proof the main theorems}
In this section, we prove the main theorems listed in the Introduction. In order to prove Theorem \ref{cobordism20}, we need the following

\begin{lemma}\label{ADM bound}\quad
Let $(M,g)$ be a 3-dimension asymptotic flat (AF) manifold with inner boundary $\partial M$ which is diffeomorphic to $S^2$, and $M$ is diffeomorphic to $\mathbb{R}^3\setminus \mathbb{B}_1$. The scalar curvature $R$ of $M$ is nonnegative and mean curvature $H$ of $\partial M$ is prescribed, then we have following
\begin{enumerate}
  \item $m_{ADM}(M)> 0$, when the Hawking mass $m_H(\partial M)>0$;
  \item $m_{ADM}(M)\ge m_H(\partial M)$, when the Hawking mass $m_H(\partial M)\le 0$.
\end{enumerate}
\end{lemma}
\begin{proof}
As $(M^3, g)$ is AF,  we can find a minimizing hull $\Sigma$ of $\p M$ in $(M^3, g)$, and then by the inverse mean curvature flow, we see that $m_{ADM}(M)\ge m_H(\Sigma)$ ( see \cite{IH}) . By the property of  minimizing hull we have
\begin{equation*}
1-\dfrac{1}{16\pi}\int_{\Sigma} H^2 d\mu=1-\dfrac{1}{16\pi}\int_{\Sigma\cap \p M}H^2 d\mu\ge 1-\dfrac{1}{16\pi}\int_{\partial M} H^2 d\mu,
\end{equation*}

note that
$$m_H(\partial M)=\dfrac{|\Sigma|^{1/2}}{(16\pi)^{1/2}}(1-\dfrac{1}{16\pi}\int_{\partial M} H^2 d\mu),$$
when $m_H(\partial M)> 0$, then $m_{ADM}(M)\ge m_H(\Sigma)> 0$.
when $m_H(\partial M)\le 0$,  since $|\Sigma|\le |\p M|$, then $m_{ADM}(M)\ge m_H(\Sigma)\ge m_H(\partial M)$.
\end{proof}

With this lemma, we are in the position to prove the following

\begin{theorem}\label{cobordism2}\quad
Suppose Bartnik data $(S^2, \gamma_i, H_i)$ satisfies $H_i>0$, $i=1,2$, and the Gaussian curvature of $(S^2, \gamma_2)$ is positive, then they admit no trivial NNSC cobordism provided
\begin{equation}
m_{BY}(S^2, \gamma_2, H_2)<0\le m_H(S^2, \gamma_1, H_1).
\end{equation}
\end{theorem}

\begin{proof}\quad
We argue it by contradiction, suppose $(S^2, \gamma_1, H_1)$ and $(S^2, \gamma_2, H_2)$ admit an NNSC cobordism $(\Omega_1,g_1)$.
By  \cite{ST}, we can construct an asymptotically flat end $(E,g)$ with $(S^2, \gamma_2)$ as the inner boundary, let $(\Sigma_0, h_0)$ be the isometric embedding image of  $(S^2, \gamma_2)$  in $\mathbb{R}^3$,  and $X$ be its
position vector,  and let $\nu$ be the unit outward normal of $\Sigma_0$ at $X$. Let $\Sigma_\rho$ be the convex hypersurface described by $Y = X+\rho \nu$, with $\rho\ge 0$. The Euclidean space outside $\Sigma_0$ can be represented by
$(\Sigma_0 \times (0,\infty), d\rho^2+h_\rho)$, $h_\rho$ is the induced metric on $\Sigma_\rho$.
Consider the following quasi-spherical equation:
\begin{equation} \label{QS}
\left\{
\begin{split}
H_0\dfrac{\p u}{\p \rho} &=u^2\Delta_\rho u+\dfrac{1}{2}(u-u^3)R_\rho\\
u(x,0) &=u_0(x)=\dfrac{H_0}{H_2}>0,
\end{split}
\right.
\end{equation}
where $R_\rho>0$ is the scalar curvature of $\Sigma_\rho$ and $H_0$ is the mean curvature of $\Sigma_\rho$ with respect to the metric $d\rho^2+h_\rho$.

By \cite{ST}, we know that equation \eqref{QS} has a unique solution such that:
\begin{description}
  \item[(1)] $u(z)=1+\dfrac{m_0}{\rho}+v,$ where $|v|=O(\rho^{-2})$ ,$|\nabla_0v|=O(\rho^{-3})$;
  \item[(2)] $g=u^2d\rho^2+h_\rho$ is asymptotically flat outside $\Sigma_0$ with scalar curvature $R=0$;
  \item[(3)] \begin{equation}\label{ADM-BY}
  16\pi m_{ADM}(g)=8\pi m_0=\lim_{\rho\rightarrow\infty}\int_{\Sigma_\rho}H_0(1-u^{-1})d\mu_\rho.
  \end{equation}
\end{description}

Moreover, due to \cite{ST}, we know the Brown-York mass on each leaf  is non-increasing, i.e,
$$m_{BY}(\Sigma_\rho, H(\rho))=\int_{\Sigma_\rho}(H_0-H(\rho))d\mu_\rho=\int_{\Sigma_\rho}H_0(1-u^{-1})d\mu_\rho$$
 is non-increasing. Then by \eqref{ADM-BY}, we have
\begin{equation}
 16\pi m_{ADM}(g)\le \int_{\Sigma_0}H_0(1-u_0^{-1})d\mu_0=m_{BY}(S^2,\gamma_2,H_2)< 0.
\end{equation}

Now we can glue $(\Omega_1,g_1)$ and $(E,g)$ along $(S^2, \gamma_2, H_2)$. Now apply Corollary \ref{Smoothing2} and the conformal deformation in section 2 by setting $g_+=g_1$ and $g_-=g$,  we can find a family of $C^2$ metrics $\tilde g_\delta=u_\delta^4 g_\delta$, such that $R_{\tilde g_\delta}\ge 0$, and
$$\lim_{\delta\rightarrow 0}m_{ADM}(\tilde g_\delta)=m_{ADM}(g)<0,$$
 i.e.,  for any $\epsilon>0$, there exists $\delta>0$, such that
 $$|m_{ADM}(\tilde g_\delta)-m_{ADM}(g)|\le \epsilon.$$

Note the mean curvature of $(S^2, \gamma_1)$ with respect to $\tilde g_\delta$ is

$$\tilde H_1=4\dfrac{\p u_\delta}{\p \nu}+H_1.$$

Then by Lemma \ref{CT} and Remark \ref{Interpolation},
\begin{equation}\label{hawking bound}
\begin{split}
m_H(\Sigma_1,\gamma_1, \tilde H_1) &=\dfrac{|\Sigma|^{1/2}}{(16\pi)^{1/2}}(1-\dfrac{1}{16\pi}\int_{S^2} \tilde H_1^2 d\mu_{\gamma_1})\\
&\ge -C\e+(1+\epsilon)m_H(\Sigma_1,\gamma_1,  H_1)\\
&\ge -C\e.
\end{split}
\end{equation}
where $C>0$ is a constant. Take $\delta$ sufficiently small, then $\epsilon$ is sufficiently small and we have $$m_H(\Sigma_1,\gamma_1, \tilde H_1)> m_{ADM}(\tilde g_\delta)$$ which is a contradiction to Lemma \ref{ADM bound}.

\end{proof}
\begin{tikzpicture}
\draw[dashed] (2,0) arc (0:180:2cm and 0.5cm);
\draw (2,0) arc (0:-180:2cm and 0.5cm);
\node[right](O) at (2,0) {$(S^2, \gamma_1, H_1)$};
\draw[dashed] (3,4) arc (0:180:3cm and 0.75cm);
\draw (3,4) arc (0:-180:3cm and 0.75cm);
\node[right](P) at (3,4) {$(S^2, \gamma_2, H_2)$};
\coordinate (A) at (-3,4);
\coordinate (B) at (3,4);
\coordinate (C) at (-2,0);
\coordinate (D) at (2,0);
\coordinate (E) at (-2.2,0);
\coordinate (F) at (-2.3,0.4);
\draw (C)--(A);
\draw (D)--(B);
\draw[->] (E)--node[left]{normal vector $\nu$}(F);
\draw[domain=-4.5:-3] plot(\x,{sqrt(3*(\x)^2-11)});
\draw[domain=3:4.5] plot(\x,{sqrt(3*(\x)^2-11)});
\node[above] at (0,6) {$(E,g)$};
\node[below] at (0,2) {$(\Omega_1,g_1)$};
\node[below] at (0,-1) {\textbf{Figure 3}};
\node[below] at (0,-1.5) { $(E,g)$ is asymptotic flat.};
\end{tikzpicture}

Next, we prove the following:

\begin{theorem}\label{2-d}\quad
Suppose Bartnik data $(S^2, \gamma_i, H_i)$ satisfies $H_i>0, i=1,2$. Then there is a  positive constant $\Lambda(\gamma_1,\gamma_2,H_1)$  depending only on $\gamma_1,\gamma_2,H_1$ such that if $(S^2, \gamma_1, H_1)$ is NNSC-cobordant to  $(S^2, \gamma_2, H_2)$, then
$$\int_{S^2} H_2 d\mu_{\gamma_2}\le \Lambda.$$
\end{theorem}

\begin{proof}\quad
For any NNSC cobordism $(\Omega_0, g_0)$  of $(S^2, \gamma_1, H_1)$ and $(S^2, \gamma_2, H_2)$, by Proposition \ref{extend} we can extend $(\Omega_0, g_0)$ to $(\Omega, g)$, the cobordism of $(S^2, \tilde\gamma_1, \tilde H_1)$ and $(S^2, \gamma_2, H_2)$ with $R_g\ge -C(\gamma_1,\gamma_2)$, where $\tilde\gamma_1$ is a round metric, $\tilde H_1$ is a positive constant(here we only need to extend the bottom face). For convenience we still denote the boundary data by $(S^2, \gamma_1, H_1)$ and $(S^2, \gamma_2, H_2)$.

 Choosing suitable $\kappa_1$ and $t_1$,  $(S^2, \gamma_1,  H_1)$ can be filled-in by a part of hyperbolic space $\mathbb{H}^3_{-\kappa_1^2}$ with metric $g_1=dt^2+\kappa_1^{-2}\sinh^2(\kappa_1 t) \gamma_{std}$, where $\gamma_{std}$ is the standard unit spherical metric on $S^2$, $ H_1=2\sqrt{1+\kappa_1^2r_1^2}/r_1$, and $r_1=\kappa_1^{-1}\sinh\kappa_1 t_1$ is the radius of $(S^2,\gamma_1)$. Denote the fill-in
by $(\Omega_1, g_1)$, now gluing $(\Omega, g)$ and $(\Omega_1, g_1)$ along $(S^2, \gamma_1)$, then we obtain a fill-in $(\Omega_2, g_2)$ of $(S^2, \gamma_2, H_2)$ with corners, where $\Omega_2=\Omega\cup\Omega_1$. Since the mean curvature of the two side of $(S^2, \gamma_1)$ are equal, by Proposition \ref{Smoothing1}, we can find some sufficiently small $\delta_0$, such that $g_{\delta_0}$ is a smooth metric on $\Omega_2$ and $R_{g_{\delta_0}}\ge -C(g,g_1)=-C(\gamma_1,\gamma_2,H_1)$ near $(S^2, \gamma_1)$. Let
$$\kappa=\max\{\kappa_1,\sqrt{C(\gamma_1,\gamma_2, H_1)/6}, \sqrt{C(\gamma_1,\gamma_2)/6},\sqrt{K_{\gamma_2-}}\},$$
where $K_{\gamma_2-}$ is the negative part of the Gaussian curvature of $(S^2,\gamma_2)$, then $(\Omega_2,g_{\delta_0})$ is a smooth manifold with scalar curvature $R_{g_{\delta_0}}\ge -6\kappa^2$.

Then $K_{\gamma_2}>-\kappa^2$ , hence $(S^2,\gamma_2)$ can be isometrically imbedded in $\mathbb{H}^3_{-\kappa^2}$. Then by \cite{ST2}, we have,
$$\int_{S^2} (H_0-H_2)\cosh \kappa r d\mu_{\gamma_2}\ge 0,$$
where $H_0$ is the mean curvature of isometric embedding image of $(S^2, \gamma_2)$ in $\mathbb{H}^3_{-\kappa^2}$ with the metric $dr^2+\kappa^{-2}\sinh^2 \kappa r\gamma_{std}$. Then
$$\int_{S^2} H_2\cosh \kappa r d\mu_{\gamma_2}\le \int_{S^2} H_0\cosh \kappa r d\mu_{\gamma_2},$$
take $\Lambda=\int_{S^2} H_0\cosh \kappa r d\mu_{\gamma_2}+1$ and note $\cosh \kappa r\ge 1$, we have
$$\int_{S^2} H_2 d\mu_{\gamma_2}\le \int_{S^2} H_2\cosh \kappa r d\mu_{\gamma_2}\le \Lambda,$$ this finishes the proof.

\end{proof}

\begin{tikzpicture}
\draw[dashed] (2,0) arc (0:180:2cm and 0.5cm);
\draw (2,0) arc (0:-180:2cm and 0.5cm);
\node[right](O) at (2,0) {$(S^2, \gamma_1, H_1)$};
\draw (3,3) arc (0:180:3cm and 0.75cm);
\draw (3,3) arc (0:-180:3cm and 0.75cm);
\node[right](P) at (3,3) {$(S^2, \gamma_2, H_2)$};
\coordinate (A) at (-3,3);
\coordinate (B) at (3,3);
\coordinate (C) at (-2,0);
\coordinate (D) at (2,0);
\coordinate (E) at (-2.2,0);
\coordinate (F) at (-2.3,0.3);
\draw (C)--(A);
\draw (D)--(B);
\draw[->] (E)--node[left]{normal vector $\nu$}(F);
\draw (-2,0) parabola bend (0,-3) (2,0);
\node[above] at (0,-2) {$(\Omega_1, g_1)$};
\node[above] at (0,1) {$(\Omega, g)$};
\node[below] at (0,-3.5) {\textbf{Figure 4}};
\node[below] at (0,-4) {$(\Omega_1, g_1)$ is a part of hyperbolic space glued to the coborsidm $(\Omega, g)$.};
\end{tikzpicture}

For the high dimensional case, we have

\begin{theorem}\label{highdimnoncob}\quad
Suppose Bartnik data $(S^{n-1}, \gamma_i, H_i)$ satisfies $H_i>0$, $i=1,2$, $3\leq n \leq 7$. Then there exists  a positive  constant $\Lambda(n,\gamma_1,\gamma_2,H_1)$ depending only on $n,\gamma_1,\gamma_2,H_1$ such that for any $H_2\ge \Lambda$,  $(S^{n-1}, \gamma_1, H_1)$ is not NNSC-cobordant to $(S^{n-1}, \gamma_2, H_2)$.
\end{theorem}
\begin{proof}\quad
We also argue it by contradiction. Suppose for any $\Lambda>0$, when $H_2\ge\Lambda$, we can find a cobordism $(\Omega_1, g_1)$. By Proposition \ref{extend}, we can assume $\gamma_1$ and $\gamma_2$ are round, $H_1$ and $H_2$ are positive constants.
 Then as in the proof of Theorem \ref{2-d}, and after smoothing the corners, we get a fill-in $(\Omega,g)$ of $(S^{n-1}, \gamma_2, H_2)$ with $R_{g}\ge -n(n-1)\kappa^2$, where $\kappa$ depends only on $\gamma_1$, $\gamma_2$ and $H_1$. By a suitable scaling, we may assume that $\kappa=1$.

Note $(S^{n-1}, \gamma_2)$ can be embedded in the standard hyperbolic space $\mathbb{H}^n$ with mean curvature $H_0=(n-1)\sqrt{1+r_2^2}/r_2$, where $r_2$ is the radius of $(S^{n-1}, \gamma_2)$ and $r_2=\sinh \rho_2$. Here the metric of $\mathbb{H}^n$ is given by $g_0=d\rho^2+(\sinh^2\rho)\gamma_{std}$.
Solve the following quasi-spherical equation:
\begin{equation} \label{QS-Hy}
\left\{
\begin{split}
H_\rho\dfrac{\p u}{\p \rho} &=u^2\Delta_\rho u+\dfrac{1}{2}(u-u^3)(R_\rho+n(n-1))\\
u_0&=\dfrac{H_0}{H_2}>0,
\end{split}
\right.
\end{equation}
we can get a unique solution $u$ and an asymptotically hyperbolic manifold $(\Omega', g')$ with an inner boundary $(\Sigma_2, \gamma_2,  H_2)$. The metric is given by $g'=u^2d\rho^2+\sinh^2\rho\gamma_{std}$ and the scalar curvature $R_{g'}=-n(n-1)$.
(see the work in \cite{ST} and \cite{WY}).

Take $\Lambda=H_0+1$, and note that  $H_2>\Lambda>H_0$, we have $u_0<1$ is a constant. Solve the equation \eqref{QS-Hy} with this initial data, we have
\begin{equation}\label{u}
u^2(\rho)=1-\dfrac{1}{1+c_0\sinh^{n-2}\rho\cosh^2 \rho}
\end{equation}
where
$$c_0=\dfrac{u_0^2}{1-u_0^2}\dfrac{1}{\sinh^{n-2}\rho_2\cosh^2 \rho_2},$$
is a constant.
Let
$$-\dfrac{1}{\sinh r}dr=ud\rho,$$

then we have
$$g=\sinh^{-2} r(dr^2+\sinh^2 r\sinh^2\rho(r)\gamma_{std}),$$
Now we compute the expansion of $\sinh^2 r\sinh^2\rho(r)$ near $r=0$ or equivalently at $\rho=+\infty$.
By \eqref{u}, we have
\begin{align}
&\lim_{\rho\rightarrow \infty}(u(\rho)-1)e^{n\rho}=-\frac{C(n)}{c_0}\\
&\lim_{\rho\rightarrow \infty}(U(\rho) -\rho)e^{n\rho}=\frac{C(n)}{nc_0},
\end{align}
where

 \begin{equation}\label{U-r}
U(\rho)=\ln \dfrac{e^r+1}{e^r-1},
\end{equation}

${C(n)}$ is a constant depend only on $n$.
Also by \eqref{U-r} and the L'Hospital law we have
\begin{align}
&\sinh U(\rho) \sinh r=1,\\
&\lim_{\rho\rightarrow \infty}\dfrac{r}{e^{-\rho}}=\lim_{\rho\rightarrow \infty}\dfrac{2u(\rho)}{e^{U(\rho)-\rho}-e^{-U(\rho)-\rho}}=2.
\end{align}

Note $\rho\rightarrow \infty$ as $r\rightarrow 0$,
\begin{equation*}
\begin{split}
\lim_{r\rightarrow 0}\dfrac{\sinh^2 r\sinh^2\rho(r)-1}{r^n} &=\lim_{r\rightarrow 0}\dfrac{\sinh^{-2} U(\rho)\sinh^2\rho-1}{r^n}\\
&=\lim_{r\rightarrow 0}\dfrac{e^{2(\rho-U(\rho))}-1}{r^n}\\
&= \lim_{r\rightarrow 0}-\frac{C(n)}{nc_0}\dfrac{e^{-n\rho}}{r^n}\\
&= -\frac{C(n)}{nc_0}.
\end{split}
\end{equation*}

Then the metric has the form
\begin{equation}
g'=\sinh^{-2} r(dr^2+\gamma_{std}-\frac{C(n)}{nc_0}r^n \gamma_{std} +O(r^{n+1})).
\end{equation}
Now gluing $(\Omega, g)$ and $(\Omega', g')$ along $(S^{n-1}, \gamma_2)$, we obtain a hyperbolic manifolds $(\Omega'', g'')$ with corners and along the corners we have $H_g\ge H_{g'}$ by our construction.
The mass aspect of $(\Omega'', g'')$ is given by
$$f_m(\Omega'', g'')= Tr_{\gamma_{std}}(-\frac{C(n)}{c_0})\gamma_{std}<0.$$
Then by the proof of Theorem 1.1 in \cite{BQ}, we may smooth the metric along $S^{n-1}$ to get a new AH metric $g$ with the scalar curvature $R_g\geq -n(n-1)$ and the mass aspect is negative  which is a contradiction to  Theorem 1.3 of \cite{ACG}. This completes the proof of Theorem  \ref{highdimnoncob}.
\end{proof}

\begin{tikzpicture}
\draw[dashed] (2,0) arc (0:180:2cm and 0.5cm);
\draw (2,0) arc (0:-180:2cm and 0.5cm);
\node[right] at (2,0) {$(S^{n-1}, \gamma_1, H_1)$};
\draw[dashed] (3,3) arc (0:180:3cm and 0.75cm);
\draw (3,3) arc (0:-180:3cm and 0.75cm);
\node[right] at (3,3) {$(S^{n-1}, \gamma_2, H_2)$};
\coordinate (A) at (-3,3);
\coordinate (B) at (3,3);
\coordinate (C) at (-2,0);
\coordinate (D) at (2,0);
\coordinate (E) at (-2.2,0);
\coordinate (F) at (-2.3,0.3);
\draw (C)--(A);
\draw (D)--(B);
\draw[->] (E)--node[left]{normal vector $\nu$}(F);
\draw (-2,0) parabola bend (0,-3) (2,0);
\node[above] at (0,1) {$(\Omega, g)$};
\draw[domain=-4.5:-3] plot(\x,{sqrt(3*(\x)^2-18)});
\draw[domain=3:4.5] plot(\x,{sqrt(3*(\x)^2-18)});
\node[below] at (0,5) {$(\Omega', g')$};
\node[below] at (0,-3.5) {\textbf{Figure 5}};
\node[below] at (0,-4) { $(\Omega', g')$ is asymptotic hyperbolic,};
\node[below] at (0,-4.5)  {$(\Omega, g)$ is the cobordism gluing with a part of hyperbolic space as in Figure 4.};
\end{tikzpicture}

\begin{theorem}\label{1}\quad
Suppose $3\leq n \leq 7$, given Bartnik data $(\Sigma_1^{n-1}, \gamma_1, 0)$ and $(\Sigma_2^{n-1}, \gamma_2, H_2)$ with $H_2>0$, and $(\Sigma_2^{n-1}, \gamma_2)$ is the boundary of some $n$-dimensional compact manifold $\Omega_2^n$, then there exists some $\Lambda(n,\gamma_2)>0$, such that for any $H_2\ge \Lambda$, $(\Sigma_1^{n-1}, \gamma_1, 0)$ and $(\Sigma_2^{n-1}, \gamma_2, H_2)$ admit no NNSC-cobordism.
\end{theorem}
\begin{proof}\quad
Take an interior point $p$ in $\Omega_2^n$, and we can construct a connected sum of $\Omega_2^n$ with an n-dimensional torus $T^n$ around $p$. Denote the resulting manifold by $\bar \Omega_2^n=\Omega_2^n\# T^n$, then $\p\bar \Omega_2^n=\p \Omega_2^n$. By the work in \cite{SWW}, we know that there is a metric $g_2$ on $\bar \Omega_2^n$ such that $R_{g_2}>0$ and $g_2|_{\Sigma_2^{n-1}}=\gamma_2$. Let $\bar H_2$ be the mean curvature of
$(\Sigma_2^{n-1}, \gamma_2)$ in $\bar \Omega_2^n$ with respect to the unit outward normal vector.

Now take $\Lambda=\max_{\Sigma_2^{n-1}}|\bar H_2|$, suppose for any $H_2\ge \Lambda$, $(\Sigma_1^{n-1}, \gamma_1, 0)$ and $(\Sigma_2^{n-1}, \gamma_2, H_2)$ admit an NNSC cobordism $(\Omega_1^n, g_1)$. Glue $(\Omega_1^n, g_1)$ and $(\bar \Omega_2^n, g_2)$ along $(\Sigma_2^{n-1}, \gamma_2)$ to form a new manifold with corners. Note here $H_2+\bar H_2\ge 0$, then by Proposition \ref{Smoothing1}, the twice conformal deformations given in section 2 and Proposition \ref{DBH2}, we can obtain a smooth manifold $(\Omega^n, g)$, such that $R_g\ge 0$ and $R_g>0$ at some point, $g|_{\Sigma_1^{n-1}}=\gamma_1$ and $H_g=0$.

Then consider the closed manifold $(\tilde \Omega^n, \tilde g)$ obtained by doubling $(\Omega, g)$ along the the minimal boundary $(\Sigma_1^{n-1}, \gamma_1)$, similar to the proof of Theorem 3 in \cite{M2}(also see \cite{LM} and \cite{ST3}), $\tilde \Omega^n$ admit a metric  $\tilde g'$ with positive scalar curvature. However, $\tilde \Omega^n=T^n\#K\# T^n$, where $K$ is an $n$-dimensional closed orientable manifold obtained by the doubling of $\Omega_2^n\cup \Omega_1^n$, then $\tilde \Omega^n$ admits no metric with positive scalar curvature by Corollary 2 in \cite{SchY2}, which is a contradiction. Thus completes the proof of Theorem \ref{1}.
\end{proof}

\Acknowledgements{This work was  supported by National Natural Science Foundation of China (National Key R$\&$D Program of China, Grant NO.11731001) and Postdoctoral Science Foundation of China (Grant No. 2020M680171).  }


\begin{thebibliography}{50}
\bahao\baselineskip 11.5pt

\bibitem{ACG} Andersson L., Cai M., and Galloway G.,  Rigidity and Positivity of Mass for Asymptotically Hyperbolic Manifolds,
Ann. Henri Poincar\text{\'{e}} 9 (2008), 1-33.

\bibitem{BMN} Brendle S., Marques F. C., Neves A.,  Deformations of the hemisphere that increase scalar curvature, Invent. Math. 185 (2011), no. 1, 175-197.

\bibitem{BQ} Bonini V. and Qing J.,  A Positive Mass Theorem on Asymptotically Hyperbolic Manifolds with Corners along a Hypersurface, Ann. Henri Poincar\text{\'{e}} 9 (2008), 347-372.

\bibitem{CM} Cabrera Pacheco A. J., Miao P.,  Higher dimensional black hole initial data
with prescribed boundary metric, Math. Res. Lett. 25 (2018), no. 3, 937-956.


\bibitem{Gromov2} Gromov M.,  Scalar curvature of manifolds with boundaries: natural questions and artificial constructions, arXiv:1811.04311v2.

\bibitem{Gromov4} Gromov M.,  Four lectures on scalar curvature, arXiv:1908.10612v3.



\bibitem{GL} Gromov M. and Lawson H.B.,  The Classification of Simply Connected Manifolds of Positive Scalar Curvature,
Ann.Math. , 111(1980), No. 3, 423-434.

\bibitem{GT} Gillbarg D. and Trudinger N.,  Elliptic partial differential equations of second oder, Berlin, Heidelberg, New York: Spinger, 2001.


\bibitem{HS}Hu X., Shi, Y.G.,  NNSC-cobordism of Bartnik data in high dimensions.,SIGMA Symmetry Integrability Geom. Methods Appl. 16 (2020), Paper No. 030, 5 pp.



\bibitem{IH} Ilmanen G. and Huisken T.,  The Inverse Mean Curvature Flow and the Riemannian Penrose Inequality,
J. Differential Geom. 59 (2001), 79-125.

\bibitem{JMT} Jauregui J.L., Miao P., Tam L.-F., Extensions and fill-ins with non-negative scalar curvature, Classical
Quantum Gravity 30 (2013), 195007, 12 pages.

\bibitem{MSch} Mantoulidis C.and Schoen R.,  On the Bartnik mass of apparent horizons, Class. Quantum Grav. 32 (2015) 205002.
\bibitem{LM}  C. Li and C. Mantoulidis,  Positive scalar curvature and skeleton singularities, Math. Ann. 374
(2019), 99-131.

\bibitem{M} Miao P.,  Positive Mass Theorem on Manifolds admitting Corners along a Hypersurface,
Adv. Theor. Math. Phys. 6 (2002) 1163-1182.

\bibitem{M2} Miao P.,  Nonexistence of NNSC fill-ins with large mean curvature,
 arXiv:2009.04976
\bibitem{Ma} Marques F.,  Deforming three manifolds with positive scalar curvature,
Ann.Math. ,(2) 176 (2012), 815-863.

\bibitem{WY} Wang M.T and Yau S.T., A generalization of Liu-Yau quasi-local mass, Commun. Anal. Geom.
Volume 15, Number 2, 249-282, 2007.

\bibitem{SchY} Schoen R., Yau S.T., On the proof of the positive mass
conjecture in general relativity. Comm. Math. Phys., 65(1):45-76, 1979.

\bibitem{SchY2} R. Schoen and S.-T. Yau,  On the structure of manifolds with positive scalar curvature, Manuscripta Math. 28 (1979), 159-183.

\bibitem{SWW} Shi Y., Wang W., Wei G.,  Total mean curvature of the boundary
and nonnegative scalar curvature fill-in,  Arxiv:2007.06756v2.

\bibitem{SWWZ} Shi Y., Wang W., Wei G., and Zhu J.,  On the Fill-in of Nonnegative Scalar Curvature Metrics,  Math. Ann. (2020), https://doi.org/10.1007/s00208-020-02087-1.

\bibitem{SWY} Shi Y., Wang W., Yu H.,  On the rigidity of Riemannian-Penrose inequality for asymptotically flat 3-manifolds with corners, Math. Z. (2019) 291:569-589.

\bibitem{ST}  Shi Y., Tam L.-F.,  Positive mass theorem and the boundary behaviors of
compact manifolds with nonnegative scalar curvature, J. Differential Geom.
62 (2002), 79-125.

\bibitem{ST2}  Shi Y., Tam L.-F.,  Rigidity of compact manifolds and positivity of
quasi-local mass,  Class. Quantum Grav. 24 (2007) 2357-2366.

\bibitem{ST3} Y. Shi and L.-F. Tam,  Scalar curvature and singular metrics, Pacific J. Math. 293 (2018), no.
2, 427-470.
\bibitem{Wa2} Walsh M.,  Aspects of positive scalar curvature and topology I, Irish Math. Soc. Bulletin, no. 80, Winter 2017, 45-68.

\bibitem{Wa3} Walsh M.,  Aspects of positive scalar curvature and topology II, Irish Math. Soc. Bulletin, no. 81, Summer 2018, 57-95.

\end{thebibliography}
\end{document}